\documentclass[12pt,psamsfonts]{amsart}
\usepackage{amsmath}
\usepackage{amsthm}
\usepackage{amssymb}
\usepackage{amscd}
\usepackage{amsfonts,mathrsfs}
\usepackage{amsbsy}
\usepackage{amssymb}
\usepackage{enumerate}
\usepackage{hyperref}
\usepackage[top=1in, bottom=1.25in, left=1.4in, right=1.4in]{geometry}
\usepackage[colorinlistoftodos,prependcaption,textsize=tiny]{todonotes}
\usepackage{xargs}
\newcommandx{\change}[2][1=]{\todo[linecolor=green,backgroundcolor=green!25,bordercolor=blue,#1]{#2}}
\newcommandx{\missing}[2][1=]{\todo[linecolor=blue,backgroundcolor=blue!25,bordercolor=red,#1]{#2}}
\newcommandx{\avis}[2][1=]{\todo[linecolor=red,backgroundcolor=red!25,bordercolor=red,#1]{#2}}
\usepackage{relsize}

\newtheorem {theorem}{Theorem}[section]
\newtheorem {proposition} [theorem]{Proposition}

\newtheorem {lemma}  [theorem]{Lemma}

\newtheorem{thmx}{Theorem}

\theoremstyle{definition}

\newcommand{\R}{\mathbb{R}}

\newcommand{\N}{\mathbb{N}}

\renewcommand{\P}{\mathscr{P}}

\begin{document}
\title[]
{Persistence of periodic traveling waves and Abelian integrals}
\author{Armengol~Gasull, Anna~Geyer  and V\'{\i}ctor~Ma\~{n}osa}

\address{Departament de Matem\`{a}tiques, Edifici Cc, Universitat Aut\`{o}noma de Barcelona, 08193 Cerdanyola del Vall\`es, Barcelona, Spain.\newline
{}
$\phantom{xx}$Centre de Recerca Matem\`atica, Edifici Cc, Campus de Bellaterra,
08193 Cerdanyola del Vall\`es, Barcelona, Spain.}
\email{gasull@mat.uab.cat}
\address{Delft Institute of Applied
Mathematics, Faculty Electrical Engineering, Mathematics and
Computer Science Delft University of Technology, Van Mourik
Broekmanweg 6, 2628 XE Delft, The Netherlands} \email{A.Geyer@tudelft.nl}
\address{Departament de Matem\`{a}tiques and 
Institut de Matem\`{a}tiques de la UPC-Bar\-ce\-lo\-na\-Tech (IMTech),
Universitat Polit\`{e}cnica
de Catalunya, Colom 11, 08222 Terrassa, Spain} \email{victor.manosa@upc.edu}




\begin{abstract}
It is well known that the existence of traveling wave
solutions (TWS) for many partial differential equations (PDE) is a
consequence of the fact that an associated planar ordinary
differential equation (ODE) has certain types of solutions defined for all
time. In this paper we address the problem of persistence of TWS of a given PDE under small perturbations. Our main results deal with the situation where the associated
ODE has a center and, as a consequence, the original PDE has a
continuum of periodic traveling wave solutions. We prove that the
TWS that persist are controlled by the zeroes of some Abelian
integrals. We apply our results to several famous PDE, like
 the Ostrovsky, Klein-Gordon, sine-Gordon, Korteweg-de Vries, Rosenau-Hyman,
 Camassa-Holm, and Boussinesq equations.
\end{abstract}

\maketitle
\setcounter{tocdepth}{2}



\section{Introduction}

  Traveling wave solutions (TWS) are an important class of
particular solutions of partial differential equations (PDE).  These
waves are special solutions which do
not change their shape and which propagate at constant speed.  They
appear in fluid dynamics, chemical kinetics involving reactions,
mathematical biology, lattice vibrations in solid state physics,
plasma physics and laser theory, optical fibers, etc. In these
systems the phenomena of dispersion, dissipation, diffusion,
reaction and convection are the fundamental physical common facts.  We refer the reader to some interesting sources to know more details about the first appearance of this kind of solutions in the works of Russell (1834), Boussinesq (1877),
Korteweg and de Vries (1895), Luther (1906),  Fisher (1937),
Kolmogorov, Petrovskii and Piskunov (1937), and to find several examples
of applications and further motivation to study them: see \cite{GasGia2015,
GilKer2004,GriSch2012,Gri1991,LiDai2007, Lie2013,Mur2002, 
SanMai1995, Zha2011} and the references therein.

When studying ordinary differential equations (ODE), especially when
they are modeling real world phenomena, it is very important to take into account whether the ODE are \emph{structurally stable}. In a few words this means that if we fix a compact set $\mathcal K$ in the
phase space it is said that an ODE is structurally stable on $\mathcal K$ when any
other close enough (in the
$\mathcal{C}^1$-topology) differential equation has a conjugated phase portrait. This
concept is relevant for applications because it implies that the
observed behaviours are qualitatively robust with respect to small changes of the model, see for instance \cite{AndroLeontGor1973, Pei1962, Sot1974} 
for more details, in particular concerning  the planar case. Recall that the boundary of the sets of structurally stable differential equations is
precisely where \emph{bifurcations} (that is, qualitative changes of 
the phase portraits) may occur.

It is well-known that for many PDE the existence of TWS is
established by proving the existence of a particular solution of a
planar ordinary differential equation.  These particular solutions
must be defined for all time and, in the light of the previous
definition, can roughly  be classified into two categories:

\begin{itemize}

\item TWS created by a dynamical behaviour that is structurally
stable. Examples of this situation are   hyperbolic limit
cycles or heteroclinic connections where both critical points are
hyperbolic and one of them is a node.

\item TWS created by a dynamical behaviour that is not structurally
stable, as for instance continua of periodic orbits,
or homoclinic or heteroclinic solutions connecting hyperbolic saddles.

\end{itemize}

In the first situation, simply take as the set $\mathcal K$ a compact
neighbourhood of the orbit that gives rise to the TWS for a given
PDE. Then it can be easily seen that a small enough $\mathcal{C}^1$ perturbation of  the original PDE with the same order will still have a TWS. This is so because all the structurally stable phenomena in ODE are robust under $\mathcal{C}^1$-perturbations.
The only condition that must be checked is that the
$\mathcal{C}^1$-closeness between the two PDE's is translated into
a $\mathcal{C}^1$-closeness in $\mathcal K$ of the corresponding
ODE.

An example corresponding to the first situation is the
Fisher-Kolmogorov PDE, $u_t=u_{xx}+u(1-u),$  where the existence of
several  TWS of front type with different speeds is associated to the existence
of a homoclinic connection between a hyperbolic saddle and a node,
see \cite{Chi2006,GasGiaTor2013} and references therein.  Therefore, all PDE of
the form $u_t=u_{xx}+u(1-u)+\varepsilon g(u,u_x,u_t,\varepsilon)$  for $\varepsilon$ small enough 
have such type of TWS. In fact, the same result
holds for many perturbed Fisher-Kolmogorov PDE with a perturbation 
term of the form 
$\varepsilon
g(u,u_x,u_t,u_{xx},u_{xt},u_{tt},\varepsilon)$.  
As  a
second example of the first situation mentioned above, for some PDE of the form $u_t=u_{xx}+h(u)u_x+g(u)$ 
there are periodic TWS  which are associated to the existence of hyperbolic limit
cycles, see for instance \cite{Cor1993,Man2003} and the references
therein.

\smallskip

In this paper we address   the second, more delicate, situation.
More specifically, we consider several  PDE having a  \emph{continuum of periodic}
TWS associated to a center of a second order ODE associated to the
PDE, and we study which conditions have to be imposed on the 
perturbation of the PDE to be able to ensure that TWS persist and to quantify them.

We split our main results into two theorems, which we state in Section \ref{se:defsres} after giving  some preliminary definitions and notations. Our first
result  deals with second order PDE, see Theorem \ref{th:main1}, and
applies to a wide range of equations. Our second result, Theorem
\ref{th:main2},  is more restrictive on the one hand because it only
considers some special perturbations, but on the other hand it applies
to higher order PDE.
In Section \ref{se:abel} we study a particular class of Abelian
integrals that will often appear in the analysis of the perturbations in Section \ref{se:appl}. For these Abelian integrals our main result is given in Theorem \ref{th:general}.
Finally, in Section \ref{se:appl} we detail some applications of our
results. First, in Section \ref{ss:m2}, we apply Theorem~\ref{th:main1} to perturbations of
 TWS of second order equations such as  the Ostrovsky, Klein-Gordon and
sine-Gordon equations. Afterwards, in Section~\ref{ss:mb2} we use
Theorem~\ref{th:main2} to study perturbations of higher order
PDE given by  the Korteweg-de Vries, Rosenau-Hyman,
 Camassa-Holm, and Boussinesq equations.

\section{Definitions and main results}
\label{se:defsres}

Consider $m$-th order partial differential equations
 of the form
\begin{equation}\label{eq:edp}
P\Big(u,\frac{\partial u}{\partial x}, \frac{\partial u}{\partial
t}, \frac{\partial^2 u}{\partial x^2}, \frac{\partial^2 u}{\partial
x\partial t},\frac{\partial^2 u}{\partial t^2},\ldots,
\frac{\partial^m u}{\partial x^m},\frac{\partial^m u}{\partial
x^{m-1}\partial t},\ldots,\frac{\partial^m u}{\partial
t^m},\varepsilon\Big)=0,
\end{equation}
where $\mathcal{W}\subset\R^{(m+1)(m+2)/2}$ is an open set,
$\mathcal{I}$ is an open interval containing $0,$
$P:\mathcal{W}\times \mathcal{I}\to\R$ is a  sufficiently smooth function and
$\varepsilon$ is a small parameter. Recall that the traveling wave
solutions  of \eqref{eq:edp} are particular solutions of the form
$u=U(x-ct)$ where $U(s)$ is defined for all $s\in\R$ and satisfies
certain boundary conditions at infinity. It is well-known that the
existence of such solutions is equivalent to finding solutions defined
for all $s$ of the $m$-th order ordinary differential equation
\begin{multline}\label{eq:pc} P_c(U,U',U'',\ldots,U^{(m)},\varepsilon):=
\\P\big(U,U',-cU',U'',-cU'',c^2U'',\ldots, U^{(m)},-c
U^{(m)},\ldots, (-c)^mU^{(m)},\varepsilon\big)=0,\end{multline}
satisfying these conditions. Here the prime denotes derivative with
respect to $s$ and $P_c:\mathcal{W}_c(\varepsilon)\times I\to\R,$
where $\mathcal{W}_c(\varepsilon)$ is an open subset of $\R^{m+1}$.

We will distinguish two cases according to whether \eqref{eq:edp} is a
second order equation ($m=2$) or a higher order equation ($m>2$).

\smallskip

\noindent{\bf Second order equations.} Our main result applies to a certain class of perturbed PDE that satisfy three conditions (i)--(iii) that we detail below. Succinctly, it requires the existence of a certain  wave speed $c\in\mathbb{R}$ such that: (i) the associated ODE has the form $U''=f_c(U,U')+\varepsilon g_c(U,U',\varepsilon)$; (ii) after a time reparameterization if necessary the planar system associated with this ODE can be written as a perturbation of a Hamiltonian system; and (iii) this Hamiltonian system has a center, and the Melnikov-Poincar\'{e}-Pontryagin function associated with the perturbation has $\ell$ simple zeroes, see \cite[Part II]{ChiLi2007} for further details.

More precisely, we will say that the PDE \eqref{eq:edp}
with $m=2$ satisfies \emph{Property}~$\mathcal{A}$ if there exists $c\in\mathbb{R}$ such that the following three
conditions hold:
\begin{enumerate}[(i)]
\item There  exist  $\mathcal{C}^1$ functions
$f_c:\mathcal{V}_c(\varepsilon)\to\R$ and
$g_c:\mathcal{V}_c(\varepsilon)\times \mathcal{I}\subset\R^3\to\R,$
with $\mathcal{V}_c(\varepsilon)\subset\R^2$ and
$\mathcal{V}_c(\varepsilon)\times \mathcal{I}\subset\R^3$ open sets,
such that,
 for $\varepsilon$ small enough,
\begin{multline*} \qquad\{(x,y)\in\mathcal{V}_c(\varepsilon)\,:\,
z=f_c(x,y)+\varepsilon g_c(x,y,\varepsilon)\}\\\subset
\{(x,y,z)\in\mathcal{W}_c(\varepsilon)\,:\,
P_c(x,y,z,\varepsilon)=0\}.
\end{multline*}
 Moreover, if $\mathcal {U}_c$ is the limit of the sets
$\mathcal{V}_c(\varepsilon)$ when $\varepsilon\to0,$ the only
solution of  $f_c(x,0)=0$ in $\mathcal {U}_c$ is $x=x_c.$

\item There exists a $\mathcal{C}^2$ function
$H_c:\mathcal{V}_c\subset\R^2\to \R^+\cup\{0\}$ such that
$H_c(x_c,0)=0,$
\[
\frac{\partial H_c(x,y)}{\partial y} = \frac{y}{s_c(x,y)},\quad
\frac{\partial H_c(x,y)}{\partial x} = -\frac{f_c(x,y)}{s_c(x,y)},
\]
for some $\mathcal{C}^1$ function $s_c:\mathcal{V}_c\subset\R^2\to \R^+.$
Notice that   $s_c$ is
such that
\[
\frac{\partial}{\partial
x}\left(\frac{y}{s_c(x,y)}\right)+\frac{\partial}{\partial
y}\left(\frac{f_c(x,y)}{s_c(x,y)}\right)\equiv 0.
\]

\item For each $h\in(0,\overline h_c),$ where $\overline h_c\in\R^+\cup\{\infty\},$ the set
\[\gamma_c(h):=\{(x,y)\in\mathcal{V}_c\,:\, H_c(x,y)=h
\}\] is a closed oval surrounding $(x_{c},0)$ and the function $M_c:
(0,\overline h_c)\to \R,$ defined as the line integral
\[
M_c(h)=\int_{\gamma_c(h)} \frac{g_c(x,y,0)}{s_c(x,y)}\, dx,
\]
has $\ell\ge 1$ simple zeroes in $(0,\overline h_c).$
\end{enumerate}

\begin{thmx}\label{th:main1} Assume that the second order PDE
\begin{equation}\label{eq:edp1}
P(u,u_t,u_x,u_{tt},u_{tx},u_{xx},\varepsilon) =0,
\end{equation}
satisfies Property $\mathcal{A}$ for some $c\in\R.$  Then:
\begin{enumerate}[(a)]
\item For $\varepsilon=0$ the PDE \eqref{eq:edp1} has a continuum of periodic TWS,
$u=U_h(x-ct),$ for $h$ in an open real interval.

\item For $\varepsilon$ small enough it has at least $\ell$ periodic
TWS, $u=U_{h_j}(x-ct,\varepsilon), j=1,2,\ldots,\ell.$
\end{enumerate}
\end{thmx}

\begin{proof}[Proof of Theorem \ref{th:main1}] From the discussion at the very beginning of this section,
a function  $U(s)$ is a TWS for the PDE \eqref{eq:edp}  if it is
defined for all time and
\begin{equation}\label{eq:fc}
P_c(U(s),U'(s),U''(s),\varepsilon)=0,
\end{equation}
where $P_c$ is defined in \eqref{eq:pc}. By using  (i) of Property
$\mathcal{A}$ we can write the above expression as
\[
U''(s)=f_c(U(s),U'(s))+\varepsilon g_c (U(s),U'(s),\varepsilon),
\]
for some suitable $f_c$ and $g_c.$ In other words,
$(x,y)=(U(s),U'(s))$ is a solution of the planar ODE
\[
  \left\{\!
   \begin{array}{l}
     x'=\dfrac{dx}{ds}=y, \\[10pt] y'=\dfrac{dx}{ds}=f_c(x,y)+\varepsilon g_c(x,y,\varepsilon).
   \end{array}
  \right.
\]
By item (ii) of Property $\mathcal{A}$ we can parameterize $U$ by a new
time, say $\tau,$ with $d\tau/d s= s_c(x,y),$ and then $x=U(\tau)$
satisfies the equivalent planar ODE
\begin{equation}\label{eq:ham-per}
 \begin{cases}
     \dot x=\dfrac{dx}{d\tau}=\dfrac{dx}{d\tau}\dfrac{d\tau}{ds}=\dfrac{y}{s_c(x,y)}=\dfrac{\partial H_c(x,y)}{\partial y},
     \\[10pt] \dot y=\dfrac{dy}{d\tau}=\dfrac{dy}{d\tau}\dfrac{d\tau}{ds}=\dfrac{f_c(x,y)}{s_c(x,y)}+\varepsilon \dfrac{g_c(x,y,\varepsilon)}{s_c(x,y)}
     =-\dfrac{\partial H_c(x,y)}{\partial x}
     +\varepsilon \dfrac{g_c(x,y,\varepsilon)}{s_c(x,y)}.
 \end{cases}
\end{equation}
When $\varepsilon=0$ the above system is Hamiltonian, and by (i) and
(iii) of Property $\mathcal{A}$ the continuum of curves $\gamma_c(h)$ for 
$0<h<\overline h$  are periodic orbits of system \eqref{eq:ham-per} with
$\varepsilon=0$ that surround the center $(x_c,0).$ The functions
$U_h(s,c)=x_h(\tau(s),c),$  where $(x_h(\tau,c),y_h(\tau,c))$ is the
parameterization of $\gamma_c(h)$, give rise to the continuum of periodic
traveling wave solutions of \eqref{eq:edp1}.

When $\varepsilon\ne0$ is small enough we are in the setting of the
perturbations of Hamiltonian systems,
\cite{ChiLi2007,DumLliArt2006}. Recall that for general perturbed
$\mathcal{C}^1$ Hamiltonian systems,
\begin{equation}\label{eq:gen}
 \begin{cases}
     \dot x=\phantom{-}\dfrac{\partial H(x,y)}{\partial y}+\varepsilon R(x,y,\varepsilon),
     \\[10pt] \dot y=-\dfrac{\partial H(x,y)}{\partial x}
     +\varepsilon S(x,y,\varepsilon),
 \end{cases}
\end{equation}
its associated Melnikov-Poincar\'{e}-Pontryagin function is
\[
M(h)=\int _{\gamma(h)} S(x,y,0)\,dx-R(x,y,0)\,dy,
\]
where the curves $\gamma(h)$ form a continuum of ovals contained in
$\{H(x,y)=h,$ for $h\in (h_0,h_1)\}.$ Then, it is known that each
simple zero $ h^*\in (h_0,h_1)$ of $M$ gives rise to a limit cycle
of \eqref{eq:gen} that tends to $\gamma(h^*)$ when $\varepsilon\to 0$.
For system \eqref{eq:ham-per}, $M(h)=M_c(h)$ and so,  each simple zero
$h_j\in(0,\overline h_c),$ $j=1,2,\ldots,\ell$ of  $M_c(h)$ gives
rise to a limit cycle of system \eqref{eq:ham-per}. Each of these
limit cycles correspond to a periodic TWS  of \eqref{eq:edp1}.
\end{proof}

\smallskip

\noindent{\bf Higher order equations ($m>2$).} In this situation our approach only
works for a particular  class of differential equations. Again,
fixed $(c,k)\in\R^2,$ we will define a property similar to Property
$\mathcal{A}$ which 
will consist of four conditions. The first one, that we will call condition (o), is the most
restrictive one and it is totally different to the ones imposed when
$m=2.$ It states that the associated ODE can somehow be reduced to a second order equation, or that some of the solutions of the associated ODE are also solutions of a related second order ODE. The rest of the conditions are quite similar to the ones of the
planar case. 

More precisely, we say that a PDE satisfies \emph{Property} $\mathcal{B}$ if there exist $c,k\in\mathbb{R}$ such that: 
\begin{enumerate}
\item [(o)] There exists a function
$Q_c:\mathcal{W}_c\times\mathcal{I}\to\R,$ where
$\mathcal{W}_c\subset\R^3$ is open and $Q_c$ is sufficiently smooth,  such that
\[
\frac{d^{m-2}}{ds^{m-2}}\big(Q_c(U,U',U'',\varepsilon)\big)=
P_c(U,U',U'',\ldots, U^{(m)},\varepsilon),
\]
where $P_c$ is defined in \eqref{eq:pc} and $U=U(s).$

\item [(i)] There  exist  two $\mathcal{C}^1$  functions
$f_{c,k}:\mathcal{V}_{c,k}(\varepsilon)\to\R$ and
$g_{c,k}:\mathcal{V}_{c,k}(\varepsilon)\times \mathcal{I}\to\R$ with
$\mathcal{V}_{c,k}(\varepsilon)\subset\R^2$ and
$\mathcal{V}_{c,k}(\varepsilon)\times \mathcal{I}\subset\R^3$ open
sets, such that,
 for $\varepsilon$ small enough,
\begin{multline*} \qquad\{(x,y)\in\mathcal{V}_{c,k}(\varepsilon)\,:\,
z=k+f_{c,k}(x,y)+\varepsilon g_{c,k}(x,y,\varepsilon)\}\\\subset
\{(x,y,z)\in\mathcal{W}_{c,k}(\varepsilon)\,:\,
Q_c(x,y,z,\varepsilon)=k\}.
\end{multline*}
 Moreover, if $\mathcal {U}_{c,k}$ is the limit of the sets
$\mathcal{V}_{c,k}(\varepsilon)$ when $\varepsilon\to0,$ the only
solution of  $f_c(x,0)=k$ in $\mathcal{V}_{c,k}$ is $x=x_{c,k}.$

\item [(ii)] There exists a $\mathcal{C}^2$ function
$H_{c,k}:\mathcal{V}_{c,k}\subset\R^2\to \R^+\cup\{0\}$ such that
$H_{c,k}(x_{c,k},0)=0,$
\[
\frac{\partial H_{c,k}(x,y)}{\partial y} =
\frac{y}{s_{c,k}(x,y)},\quad \frac{\partial H_{c,k}(x,y)}{\partial
x} = -\frac{f_c(x,y)+k}{s_{c,k}(x,y)},
\]
for some $\mathcal{C}^1$ function $s_{c,k}:\mathcal{V}_{c,k}\subset\R^2\to
\R^+$. Notice that the
function $s_{c,k}$ is such that
\[
\frac{\partial}{\partial
x}\left(\frac{y}{s_{c,k}(x,y)}\right)+\frac{\partial}{\partial
y}\left(\frac{f_c(x,y)+k}{s_{c,k}(x,y)}\right)\equiv0.
\]

\item [(iii)] For each $h\in(0,\overline h_{c,k}),$ where $\overline h_{c,k}\in\R^+\cup\{\infty\},$  the set
\[\gamma_{c,k}(h):=\{(x,y)\in\mathcal{V}_{c,k}\,:\, H_{c,k}(x,y)=h
\}\] is a closed oval surrounding $(x_{c,k},0)$ and the function
$M_{c,k}: (0,\overline h_{c,k})\to \R,$ defined as the line integral
\[
M_{c,k}(h)=\int_{\gamma_{c,k}(h)}
\frac{g_{c,k}(x,y,0)}{s_{c,k}(x,y)}\, dx,
\]
has $\ell\ge1$ simple zeroes in $(0,\overline h_{c,k}).$
\end{enumerate}

\begin{thmx}\label{th:main2} Assume that the $m$-th order PDE \eqref{eq:edp}, with $m>2,$ satisfies Property $\mathcal{B},$ for some
$c\in\R$ and $k\in\R.$  Then:
\begin{enumerate}[(a)]
\item For $\varepsilon=0$ the PDE \eqref{eq:edp1} has a continuum of periodic TWS,
$u=U_{h,k}(x-ct)$ for $h$ in an open real interval.

\item For $\varepsilon$ small enough it has at least $\ell$ periodic
TWS, $u=U_{h_j,k}(x-ct,\varepsilon)$ for $j=1,2,\ldots,\ell.$
\end{enumerate}
\end{thmx}

\begin{proof}[Proof of Theorem \ref{th:main2}]
From  condition (o) of  Property $\mathcal{B},$  if we restrict our
attention to the solutions of \eqref{eq:pc} contained in
\begin{equation}\label{eq:qc}
Q_c(U,U',U'',\varepsilon)=k,
\end{equation}
for the given value of $k\in\R,$ we can find some TWS with speed $c$
and associated to this particular value of $k.$ Other values of $k$
give different TWS with the same speed.

Starting from equation \eqref{eq:qc}, instead of equation
\eqref{eq:fc}, we can repeat all the steps of the proof of Theorem
\ref{th:main1}, point by point, to get the desired conclusion.
\end{proof}

\section{Some particular Abelian integrals}\label{se:abel}

This section is devoted to studying a particular class of Abelian
integrals for which we prove a result quantifying their zeros, see Theorem \ref{th:general}. We will use this result in the next section when we study the
persistence of TWS for several perturbed PDE, which is governed by the number of zeros of integrals of this type.

\begin{proposition}\label{pr:zeroes} Let $A,B$ and $D$ be analytic
functions, defined in an open interval   $\mathcal{I}\subset\R$ and
such that
\begin{align*}
A(x)=&a^2+O(x-x^*),\qquad B(x)=\frac{(x-x^*)^2}{b^2}+O\big((x-x^*)^3\big),\\
D(x)=&(x-x^*)^{2n}D_0(x)\quad\mbox{with}\quad D_0(x)=d+O(x-x^*),
\end{align*}
for certain $x^*\in\mathcal{I},$ where $a,b,c$ are real constants with $abd\ne0$ and
$n\in\N\cup\{0\}.$ Consider the Hamiltonian function
$H(x,y)=A(x)y^2+B(x).$ Then, the following holds:
\begin{enumerate}[(a)]
\item The Hamiltonian system $\dot x=H_y(x,y),$ $\dot y=-H_x(x,y),$
has a center at $(x^*,0).$ We will denote by $\gamma(h)$ the
periodic orbits contained in $\{H(x,y)=h\},$ which exist when
$h\in(0,\widetilde h)$ for some $\widetilde{h}\in\R.$

\item For $h\in(0,\widetilde h)$ and $p,n\in\N,$ define the Abelian integral
\begin{equation}\label{eq:integralJ}
J_p(h)=\int_{\gamma(h)}  D(x) y^p \,dx.
\end{equation}
Then $J_{2p}(h)\equiv0$ and
\[
J_{2p-1}(h)\sim \frac{2d b^{2n+1}}{a^{2p-1}}
\frac{(2p-1)!!\,(2n-1)!!\,\pi}{2^{p+n}(p+n)!}\,h^{p+n} \,
\mbox{ at }\, h=0^+,\] where $(2k-1)!!=(2k-1)(2k-3)\cdots 3\cdot 1$
and $(-1)!!=1!!=1.$
\end{enumerate}

\end{proposition}

\begin{proof} 

Without loss of generality we will assume that
$x^*=0.$ To prove $(a)$, notice that the origin is a non-degenerate singular point of the vector field $X=(H_y,-H_x)$ because $\det(\mathrm{D}X(0,0))=2A(0)B''(0)=4a^2/b^2>0$. Moreover, since a singular point of a Hamiltonian system can neither be a focus nor a node, it is a center.

To study the Abelian integral $J_p$ it is convenient to introduce
the new variable $w$ as $h=w^2.$ Then, by the Weierstrass preparation
theorem, see for instance \cite{AndroLeontGor1973,ChoHal1982}, in a
neighbourhood of $(0,0)$ the only solutions of equation
$B(x)-w^2=x^2/b^2-w^2+O(x^3)=0$ are
 \[
x=x^{\pm}(w)= \pm b w +O(w^2),
\]
where $x^{\pm}(w)$ are analytic functions at zero. Moreover, in this
neighbourhood,
\begin{equation}\label{eq:wpt}
 w^2-B(x)=(x-x^-(w))(x^+(w)-x) U(x,w),
\end{equation}
where $U(0,0)=1/b^2$ is also analytic at $(0,0).$ Notice that the points of the 
oval $\gamma(h)$ satisfy $y=\pm\sqrt{(w^2-B(x))/A(x)}.$  When $p$ is even the integral 
\eqref{eq:integralJ} vanishes because of symmetry with respect to $y=0$. Hence 
\begin{equation*}
J_p(w^2)=
\begin{cases}
0, \quad &\quad\mbox{when}\quad p\quad\mbox{is even},\\[3pt]
2\mathlarger{\int_{x^-(w)}^{x^+(w)}}
D(x)\left(\dfrac{w^2-B(x)}{A(x)}\right)^{\frac p 2}\, dx,
&\quad\mbox{when}\quad  p\quad\mbox{is odd}.
\end{cases}
\end{equation*}

By using \eqref{eq:wpt}
we get that
\begin{align*}
J_{2p-1}(w^2)=&2\int_{x^-(w)}^{x^+(w)}
D(x)\left(\dfrac{w^2-B(x)}{A(x)}\right)^{\frac{2p-1}2}\,
dx\\=&2\int_{x^-(w)}^{x^+(w)}
\big((x-x^-(w))(x^+(w)-x)\big)^{\frac{2p-1}2}
D(x)\left(\dfrac{U(x,w)}{A(x)}\right)^{\frac{2p-1}2}\, dx\\=& 2
(\Delta(w))^{2p}\int_0^1  \big(z(1-z)\big)^{\frac{2p-1}2}
\overline{D}(z,w)\left(\dfrac{\overline{U}(z,w)}{\overline{A}(z,w)}\right)^{\frac{2p-1}2}\,
dz,
\end{align*}
where in the integral we have introduced the change of variables
$z=(x-x^-(w))/\Delta(w),$  being $\Delta(w)=x^+(w)-x^-(w),$ and for
any function $E(x,w)$  or $E(x),$ we denote $\overline E(z,w)=
E\big(\Delta(w)z+x^-(w),w\big)$ or $\overline E(z,w)=E(\Delta(w)z+x^-(w))$. In particular,
\[
\overline
D(z,w)=\big(\Delta(w)z+x^-(w)\big)^{2n}\overline{D}_0(z,w)=(\Delta(w))^{2n}\left(z+\frac{x^-(w)}{\Delta(w)}\right)^{2n}\overline{D}_0(z,w),
\]
with $\overline{D}_0(0,0)=d.$ Hence,
\[
J_{2p-1}(w^2)=(\Delta(w))^{2p+2n}\int_0^1
\big(z(1-z)\big)^{\frac{2p-1}2} F(z,w)\, dz,
\]
where
\[
F(z,w)=
2\overline{D}_0(z,w)\left(z+\frac{x^-(w)}{\Delta(w)}\right)^{2n}\left(\dfrac{\overline{U}(z,w)}
{\overline{A}(z,w)}\right)^{\frac{2p-1}2}.
\]
Since $x^{\pm}(w)= \pm b w +O(w^2),$ it holds that $\Delta(w)=2b
w+O(w^2)$ and hence $\lim_{w\to0} \frac{x^-(w)}{\Delta(w)}=-\frac
12.$ Therefore for all $z\in[0,1]$ and $w$ small enough the function
$F(z,w)$  is continuous, and as a consequence
\begin{align*}
\lim_{w\to0}\frac{J_{2p-1}(w^2)}{w^{2p+2n}}&=\lim_{w\to0}\left(\frac{\Delta(w)}{w}\right)^{2p+2n}\int_0^1
\big(z(1-z)\big)^{\frac{2p-1}2} \lim_{w\to0} F(z,w)\, dz\\&=
(2b)^{2p+2n}\frac{2 d}{a^{2p-1} b^{{2p-1}}} \int_0^1
\big(z(1-z)\big)^{\frac{2p-1}2}\Big(z-\frac12\Big)^{2n}\,dz.
\end{align*}
Now we claim that
\[K(p,n):=
\int_0^1 \big(z(1-z)\big)^{\frac{2p-1}2} \Big(z-\frac12\Big)^{2n}\,
dz= \frac{(2p-1)!!\,(2n-1)!!}{8^{p+n}(p+n)!}\pi,
\] and we observe that, from this claim, 
the result follows. 

To prove the claim we observe that by using integration by parts, one easily gets that
\begin{equation}\label{e:parts}
K(p,n)=\frac{2p-1}{2n+1}K(p-1,n+1).
\end{equation} Now the claim follows by using induction. First we prove that for any $p\in\mathbb{N}_0$, $K(p,0)$ satisfies the claim. Indeed, if $\mathrm{B}$ is the Euler's Beta function, and since $ \mathrm{B}(x,y)=\Gamma(x)\Gamma(y)/\Gamma(x+y)$, we have
$$
K(p,0)=\int_0^1 z^{p-\frac{1}{2}}(1-z)^{p-\frac{1}{2}}dz= \mathrm{B}\left(p+\frac{1}{2},p+\frac{1}{2}\right)=\frac{(\Gamma(p+\frac{1}{2}))^2}{\Gamma(2p+1)}.
$$ 
Since $p$ is an integer number $\Gamma(2p+1)=(2p)!$. On the other hand, it is well known that $$
\Gamma\left(p+\frac{1}{2}\right)=\frac{(2p)!}{4^p\, p!}\sqrt{\pi}=\frac{(2p-1)!!}{2^p}\sqrt{\pi}.
$$ 
Hence, 
$$
K(p,0)= \frac{(2p-1)!!}{8^p\, p!}\pi
$$as we wanted to prove.

Now we assume that for $n>0$ and for all $p\in\mathbb{N}_0$, $K(p,n)$ satisfies 
the claim. By using the relation \eqref{e:parts}, we get that,
$$K(p,n+1)=\frac{2n+1}{2p+1}K(p+1,n)=\frac{(2n+1)\,(2p+1)!!\,(2n-1)!!}{(2p+1)\,8^{p+n+1}(p+n+1)!}\pi=$$
$$=\frac{(2p-1)!!\,(2n+1)!!}{8^{p+n+1}(p+n+1)!}\pi,$$
so the claim follows.
\end{proof}

Before proving the main result of this section, Theorem \ref{th:general}, and to motivate one of its hypotheses, we collect some simple
observations in the following lemma.

\begin{lemma}\label{le:lemma} Let $\gamma(h)\subset\{H(x,y)=h\}$, $h\in(0,\overline h)=\mathcal{L},$ be a continuum of
periodic orbits surrounding a center, corresponding to $h=0,$  of
the Hamiltonian system associated to a  $\mathcal{C}^1$ Hamiltonian function
$H(x,y)$ and assume that they have  a \emph{clockwise} time
parameterization. For  each $p,q\in\N\cup\{0\},$ consider the
Abelian integral
\[
J_{q,p}(h)=\int_{\gamma(h)} x^q y^p\,dx.
\]
The following holds.
\begin{enumerate}[(a)]
\item When $q$ is even and $p$ is odd then $J_{q,p}(h)>0$ for all
$h\in\mathcal{L}.$

\item When $q$ is even and $p$ is even and $H(x,y)=H(x,-y)$ then $J_{q,p}(h)\equiv0$
on the whole interval $\mathcal{L}.$

\item When $q$ is odd and $p$ is odd and $H(-x,y)=H(x,y)$ then $J_{q,p}(h)\equiv0$
on $\mathcal{L}.$
\end{enumerate}
\end{lemma}
\begin{proof} Notice that by Green's  theorem 
\[ 
J_{q,p}(h)=\int_{\gamma(h)} x^q
y^p\,dx=\iint_{\operatorname{Int}(\gamma(h))} p x^q y^{p-1}\,
dx\,dy,
\]
where $\rm{Int}(\gamma(h))$ denotes the interior of the oval. 
Then, trivially $(a)$ follows. The other two statements are
consequence of the symmetries of $H$ and the function $x^q y^{p-1}.$
\end{proof}

The next proposition will be one of the key results to prove  Theorem \ref{th:general}, which is stated below.

\begin{proposition}\label{pr:li} {\rm (\cite{ColGasPro})} Set $\mathcal{L}\subset \R$ an open real interval and let
$F_j:\mathcal{L}\to\R,$ $j=0,1,\ldots,\ell,$ be $\ell+1$ linearly
independent analytic functions. Assume also that one of them, say
$F_k, 0\le k\le \ell,$  has constant sign on $\mathcal{L}.$ Then,
there exist real constants $d_j,$ $j=0,1,\ldots,\ell,$ such that the
linear combination $\sum_{j=0}^\ell d_j F_j$ has at least $\ell$
simple zeroes in ${\mathcal L}.$
\end{proposition}

Notice that in the next theorem, and due to Lemma \ref{le:lemma}, the
monomials of the Abelian integral that we consider are of the form
$x^{2q}y^{2p-1}.$

\begin{thmx}\label{th:general} Let $H(x,y)=A(x)y^2+B(x),$ with
$A$ and $B$ functions satisfying the hypotheses of Proposition
\ref{pr:zeroes}, and denote by $\gamma(h),$ $h\in(0,\overline h),$
the periodic orbits surrounding the origin of the corresponding Hamiltonian system. For $d_0,d_1,\ldots,d_n\in\R$ and
$q_j,p_j\in\N,j=0,1,\ldots,\ell$ consider the family of Abelian
integrals
\[
J(h)=\int_{\gamma(h)} \sum_{j=0}^\ell d_j x^{2q_j} y^{2p_j-1}\,dx.
\]
If all values $m_j=q_j+p_j, j=0,1,\ldots,\ell$ are different, there
exit values of $d_j,j=0,1,\ldots,\ell,$ such that the corresponding
function $J(h)$ has at least $\ell$ simple zeroes in $(0,\overline
h).$
\end{thmx}

\begin{proof} Notice that
\[
J(h)=\sum_{j=0}^\ell d_j J_j(h),\quad\mbox{where}\quad J_j(h)=
\int_{\gamma(h)}x^{2q_j} y^{2p_j-1}\,dx.
\]
By Proposition \ref{pr:zeroes}, for each $j=0,1,\ldots, \ell,$
$J_j(h)=k_j h^{m_j}+o\big(h^{m_j}\big)$ and, by hypothesis,  all
these $m_j$ are different. This clearly implies that all these
$\ell+1$ functions are linearly independent. Moreover, by item $(a)$
of Lemma \ref{le:lemma} we know that none of them vanish in
$(0,\overline h).$ Hence we can apply Proposition \ref{pr:li} to
this set of functions and $\mathcal{L}=(0,\overline h)$ and  the
result follows.
\end{proof}

\section{Applications}\label{se:appl}

In this section we consider perturbations of  several relevant PDE with continua of
periodic TWS and prove that the perturbations can be tailored such that a prescribed number of TWS persist in these perturbed PDE. In many examples, for simplicity, we perturb the PDE with an
additive term that only contains partial derivatives up to $m-1.$
For more general perturbations, even including terms of order $m,$
most of the results can be adapted.

\subsection{Second order PDE}\label{ss:m2}

We start with an illustrative toy example for which we give all the details on how a prescribed number of periodic TWS can be obtained.

\subsubsection{A toy example}
Consider the PDE
\begin{equation}\label{eq:toy}
u+au_{xx}+bu_{xt}+du_{tt}+\varepsilon g(u_x,u_t,\varepsilon)=0,
\end{equation}
with $g$ a $\mathcal {C}^1$ function and take $c$ such that
$a-bc+dc^2>0.$ Then equation \eqref{eq:pc} can be written as
\[
U+(a-bc+dc^2)U''+\varepsilon g(U',-cU',\varepsilon)=0.
\]
We define $C^2=a-bc+dc^2$ and
$g_c(U',\varepsilon)=-g(U',-cU',\varepsilon)/C^2.$ Then it is easy to
see that this PDE satisfies Property $\mathcal{A}$ with $H_c(x,y)=
x^2/(2C^2)+y^2/2,$ $s_c(x,y)\equiv 1$ and $(0,\overline
h)=(0,\infty).$ That is,
$$
 \begin{cases}
     \dot x=\phantom{-}\dfrac{\partial H_c(x,y)}{\partial y}=y,
     \\[10pt] \dot y=
     -\dfrac{\partial H_c(x,y)}{\partial x}
     +\varepsilon g_c(x,y,\varepsilon)=-\dfrac{x}{C^2} +\varepsilon g_c(y,\varepsilon).
 \end{cases}
$$
Moreover
\[
M_c(h)=\int_{\gamma(h)}  g_c(y,0)\, dx,
\]
where $\gamma(h)$ is the ellipse $\{x^2/(2C^2)+y^2/2=h\}.$ We
parameterize the closed curves $H_c(x,y)=h$ as $(x,y)= (C\sqrt{2h}
 \cos \theta, \sqrt{2h} \sin \theta)$  for $0\le \theta\le 2\pi.$
Then
\[
M_c(h)= -C\sqrt{2h} \int_0^{2\pi}
g_c(\sqrt{2h}\sin\theta,0)\sin\theta\,d\theta.
\]
Assume for instance that $g_c(y,0)=\sum_{j=0}^N g_j y^j$ is a
polynomial of degree $N$, and $g_j\in\mathbb{R}$. Then,
\[
M_c(h)=  -C\sqrt{2h} \sum_{j=0}^N g_j (\sqrt{2h})^{j}\Big(\int_0^{2\pi}
 \sin^{j+1} \theta \,d\theta\Big).
\]
When $j$ is even, by symmetry, the above integrals vanish. Hence
\[
M_c(h)=-2 C h\Big( \sum_{i=0}^{[(N-1)/2]} g_{2i+1} 2^i I_{2i+2}
h^{i} \Big),
\]
where $[\,\cdot\,]$ denotes the integer part and $I_{2n}
=\int_0^{2\pi} \sin^{2n}\theta\,d\theta>0.$ Removing the factor $h$,
and taking suitable
$g_{2i+1}$, the polynomial  $M_c(h)/h$ can be
any arbitrary polynomial of degree $[(N-1)/2]$ in $h$. Hence, by applying
Theorem \ref{th:general}, for any $\ell\le [(N-1)/2],$ there exist coefficients $g_j$ such that the function $M_c(h)$ has $\ell$ simple zeros and, therefore, by applying
Theorem \ref{th:main1}, the PDE \eqref{eq:toy} has at least~$\ell$
periodic TWS.  We remark that the above computations are essentially
the same as the ones of the celebrated paper \cite{LinMelPug1977}
where the authors present the first example of classical polynomial
Li\'{e}nard differential system of degree $N$ with $[(N-1)/2]$ limit
cycles.

By doing similar computations we can consider more general
perturbations in PDE \eqref{eq:toy}, like for instance
\[
u+au_{xx}+bu_{xt}+du_{tt}+\varepsilon ( u u_{xx}+ g(u,
u_x,u_t,\varepsilon))=0,
\]
and similar results hold.

\subsubsection{Reduced Ostrovsky equation.}

We consider perturbations of the reduced Ostrovsky equation, introduced by L.~Ostrovsky in 1978,  which is a modification of the Korteweg-de Vries equation that models gravity waves propagating in a rotating background under the influence of the Coriolis force when the high-frequency dispersion is neglected.  More
concretely, we take

\begin{equation}\label{eq:bur}
(u_t+uu_x)_x-u+\varepsilon g(u,u_x,u_t,\varepsilon)=0,
\end{equation}
which satisfies Property $\mathcal{A}$ with $c>0,$ because its associated
ODE is
\[
(U-c)U''+(U')^2-U+\varepsilon g(U,U',-cU',\varepsilon)=0.
\]
Then, taking $g_c(U,U',\varepsilon)=-g(U,U',-cU',\varepsilon)/(U-c);$ 
$\mathcal{V}_c=\{x<c\};$  $x_c=0;$  and $s_c(x,y)=(x-c)^{-2}$, 
the system that has to be studied to find TWS is
$$
 \begin{cases}
     \dot x=\phantom{-}\dfrac{\partial H_c(x,y)}{\partial y},
     \\[10pt] \dot y=
     -\dfrac{\partial H_c(x,y)}{\partial x}
     +(x-c)^2\varepsilon g_c(x,y,\varepsilon).
 \end{cases}
$$
with
\[H_c(x,y)=\frac{(x-c)^2y^2}2+\frac {c x^2} 2 -\frac{x^3}3.
\]
Consider also the Melnikov-Poincar\'{e}-Pontryagin function
\[
M_c(h)=\int_{\gamma_c(h)}(x-c)^2 g_c(x,y,0) \, dx,\, h\in(0,c^3/3).
\]
As in the toy example, it is not difficult to find a 
perturbation term $g$ such that the 
the function $M_c(h)$ has
several simple zeroes in $(0,c^3/3)$ which, by Theorem~\ref{th:general},
give rise to periodic TWS of the PDE \eqref{eq:bur}.

\smallskip

\subsubsection{Perturbed non-linear Klein-Gordon equation.} The
 Klein-Gordon equation is a wave equation related to the Schr\"odinger equation, which is used to model spinless relativistic particles. It was introduced in 1926 in parallel by O. Klein,  W. Gordon and
V. Fock as a tentative to describe the relativistic electron dynamics. 
In the one-dimensional setting we look at a perturbation of this equation of the form
 \[
u_{tt}-u_{xx}+\lambda u^p+\varepsilon g(u,u_x,u_t,\varepsilon)=0,
\]
with $\lambda\in\R^+$ and $p$ an \emph{odd} integer. It can readily be seen
that it satisfies Property~$\mathcal{A}$ and  that the system that has to be
studied to find TWS is
\begin{equation*}
 \begin{cases}
     \dot x=\phantom{-}\dfrac{\partial H_c(x,y)}{\partial y},
     \\[10pt] \dot y=-\dfrac{\partial H_c(x,y)}{\partial x}
     +\varepsilon g_c(x,y,\varepsilon)  ,
 \end{cases}
\end{equation*}
where \[
H_c(x,y)=\frac{C x^{p+1}}{p+1}+\frac {y^2}2,
\]
with  $C=\lambda/(c^2-1),$
$g_c(x,y,\varepsilon)=-g(x,y,-cy,\varepsilon)/(c^2-1).$
The associated Melnikov-Poincar\'{e}-Pontryagin function is
\[
M_c(h)=\int_{\gamma_c(h)}g_c(x,y,0) \, dx,\quad h\in(0,\infty).
\]
The interested reader can take a look to the papers
\cite{CimGasMan1997, LiLiLliZha2001} where perturbations of this
Hamiltonian system and the zeros of its associated Melnikov-Poincar\'{e}-Pontryagin function 
are studied with two different approaches. 

In particular,
the zeroes of the above first integral can be studied in a similar way to the toy example
 considered at the beginning of this section. 
Notice, however, that when $p\geq 3$, instead of using trigonometric functions to parametrize the invariant closed curves, one can use the generalized polar coordinates introduced by Lyapunov in 1893 in his study of the stability of degenerate critical points,~\cite{Liapunov}.  All the details can be found in  \cite{CimGasMan1997}. Again, Theorem \ref{th:main1} guarantees that the zeros of the function $M_c(h)$ correspond with periodic TWS of the Klein-Gordon equation.

\smallskip

\subsubsection{Perturbed sine-Gordon equation.}

The sine-Gordon equation first appeared in 1862 in the context of differential geometry. Specifically in a study by E. Bour on surfaces of constant negative curvature. The equation was rediscovered later by J. Frenkel and
T. Kontorova in 1939, in their study of crystal dislocations. The equation is relevant to the community investigating integrable systems because it has soliton solutions.
Its perturbation
writes as
\[
u_{tt}-u_{xx}+\sin u+\varepsilon g(u,u_x,u_t,\varepsilon)=0.
\]
Again, it satisfies Property $\mathcal{A}$ for $c>1,$ and its associated
planar system is
\begin{equation*}
 \begin{cases}
     \dot x=\phantom{-}\dfrac{\partial H_c(x,y)}{\partial y},
     \\[10pt] \dot y=-\dfrac{\partial H_c(x,y)}{\partial x}
     +\varepsilon g_c(x,y,\varepsilon)  ,
 \end{cases}
\end{equation*}
where $$H_c(x,y)={C(1-\cos x)}+\frac{y^2}{2},$$
with  $C=1/(c^2-1)>0$ and
$g_c(x,y,\varepsilon)=g(x,y,-cy,\varepsilon)/(1-c^2).$
The Melnikov-Poincar\'{e}-Pontryagin function is
\[
M_c(h)=\int_{\gamma_c(h)}g_c(x,y,0) \, dx,\quad h\in(0,2C).
\]
The above type integrals are studied for instance in
\cite{GasGeyMan2016}. There, several condition on $g$ for obtaining
many simple zeroes of $M_c$, and therefore periodic TWS of the considered PDE, are obtained.

\subsection{PDE with order greater than 2}\label{ss:mb2}

In this section we study perturbations of several PDE with order
$m>2.$  We start with the following result that helps us to characterize the existence of centers for the unperturbed Hamiltonian systems that will appear.

\begin{lemma}\label{l:lemacentre}
Consider a Hamiltonian system of the form
$$
 \begin{cases}
     \dot x=\dfrac{\partial H(x,y)}{\partial y}=y\,m(x,y),
     \\[10pt] \dot y=-\dfrac{\partial H(x,y)}{\partial x}
=f(x,y)\,m(x,y),
 \end{cases}
$$   where $H\in \mathcal{C}^2$,  $m(x,y)>0$ and such that
$\frac{\partial}{\partial x}\left(y m(x,y)\right)+\frac{\partial}{\partial y}\left(f(x,y)m(x,y)\right)\equiv0.$ Then, a singular point of the form $(x_*,0)$ is a center if 
\begin{equation}\label{eq:condlema}
\frac{\partial}{\partial x}f(x,y)\Big|_{(x_*,0)}<0.
\end{equation} 

Furthermore, if $m(x,y)$ depends only on $x$, condition \eqref{eq:condlema} holds, and $H$ is analytic, then the Hamiltonian $H$ satisfies the hypotheses of Proposition \ref{pr:zeroes}.
\end{lemma}

\begin{proof} Consider the vector field $X=(H_y,-H_x)$. Since
$$\det(\mathrm{D}X(x_*,0))=-m^2(x_*,0)\,\frac{\partial}{\partial x}\left(f(x,y)\right)\Big|_{(x_*,0)},
$$  then equation \eqref{eq:condlema} implies that 
$\det(\mathrm{D}X(x_*,0))>0$ and therefore  $(x_*,0)$ is a center (once more, remember that a singular point of a Hamiltonian system cannot be neither a focus nor a node).

If $m(x,y)=m(x)$, then $H(x,y)=y^2\,m(x)/2+B(x)$ for some analytic function~$B.$ Since $m(x)>0$ we can write $m(x)=2a^2+O(x-x_*)$ near $x=x_*.$ Suppose that condition \eqref{eq:condlema} holds, then $H_{xx}(x_*,0)=-f_x(x_*,0)m(x_*)>0$, and we can write
$1/b^2=B''(x_*)=H_{xx}(x_*,0)$, obtaining $B(x)= (x-x_*)^2/b^2+O\big((x-x_*)^3\big)$. So $H$ fulfills the hypotheses of Proposition \ref{pr:zeroes}.
\end{proof}

Observe that condition \eqref{eq:condlema} is equivalent to the fact that $H_{xx}(x_*,0)>0$ and $\det(\mathbf{H}_H(x_*,0))>0$ (where $\mathbf{H}$ is the hessian matrix), which implies that $H$ has a non-degenerate local minimum at $(x_*,0)$.

\subsubsection{Perturbed generalized Korteweg-de Vries equation}\label{sss:kdv} 
We consider a perturbation of a family of PDE which for certain values of
the parameters contains the celebrated Korteweg-de Vries and Benjamin-Bona-Mahony equations appearing in several domains of physics (non-linear mechanics, water
waves, etc.). More concretely, we consider the family of PDE
\begin{multline}\label{eq:kdv}
u_t+au_x+buu_x+duu_t+pu_{xxx}+qu_{xxt}+ru_{xtt}+su_{ttt}\\+\varepsilon
\nabla g(u,u_x,u_t,\varepsilon)\cdot (u_x,u_{xx},u_{xt},0)^t=0.
\end{multline}
Notice that the KdV equation corresponds to
$\varepsilon=0$ and $a=d=q=r=s=0,$ $b=-6$  and $p=1.$
The ODE  associated to \eqref{eq:kdv} is
\[
\Big((a-c) U+ \frac{b-dc}2{U^2}+C U''+\varepsilon
g(U,U',-cU',\varepsilon) \Big)'=0,
\]
where $C=p-qc+rc^2-sc^3.$  Notice that then,  for any function $U$
satisfying previous equation, it holds that  there exists  $k\in\R,$
such that
\begin{equation}\label{eq:rep}
(a-c) U+ \frac{b-dc}2{U^2}+C U''+\varepsilon
g(U,U',-cU',\varepsilon)=k
\end{equation}
Thus we have to study the equivalent planar system
\begin{equation*}
 \begin{cases}
     \dot x=y= \dfrac{\partial H_{c,k}(x,y)}{\partial y},
     \\[10pt] \dot y=\alpha_{c,k}+\beta_c x+ \gamma_c x^2  +\varepsilon g_c(x,y,\varepsilon)=  -\dfrac{\partial H_{c,k}(x,y)}{\partial x}
     +\varepsilon g_c(x,y,\varepsilon),
 \end{cases}
\end{equation*}
where
\[
H_{c,k}(x,y)=-\alpha_{c,k}x-\frac{\beta_c}2 x^2-\frac{\gamma_c}3 x^3
+\frac{1}2 y^2,\]
with
\[
\alpha_{c,k}=\frac{k}C,\quad \beta_c=\frac{c-a}C,\quad
\gamma_c=\frac{dc-b}C,
\]
and  $g_c(x,y,\varepsilon)=-g(x,y,-cy,\varepsilon)/C.$  Hence, using Lemma \ref{l:lemacentre}, it is not difficult to see that the PDE 
\eqref{eq:kdv} satisfies Property $\mathcal{B}$ when equation
$\alpha_{c,k}+\beta_c x+ \gamma_c x^2=0$ has two different real
solutions (that correspond to a center and a saddle of the planar
system). Then, by Theorem \ref{th:main2}, the periodic TWS that persist for
$\varepsilon$ small enough correspond to the simple zeroes of the
elliptic integral
\[
M_{c,k}(h)=\int_{\gamma_{c,k}(h)} g_c(x,y,0)\, dx
\]
in a suitable open interval of energies. 
This kind of Abelian integrals are
studied in detail in the classical paper of Petrov
(\cite{Petrov1988}) and more recently in the Chapter 3 of Part II of the book \cite{ChiLi2007}.
Again, it is not difficult to impose conditions on $g$ to get a prescribed number of 
TWS for \eqref{eq:kdv} for $\varepsilon$ small enough and different
values of $c$ and $k.$


\smallskip

\subsubsection{Perturbed Rosenau-Hyman equation} The Rosenau-Hyman equation is a generalization of the KdV equation. It was introduced in 1993 by P. 
Rosenau and J.M. Hyman to show the existence of solitary waves
with compact support (compactons) in the context of non-linear dispersive equations.
We consider the
perturbed equation
\[
u_t+a(u^n)_x+(u^n)_{xxx}+\varepsilon \nabla
g(u,u_x,u_t,\varepsilon)\cdot (u_x,u_{xx},u_{xt},0)^t=0,
\]
where $a\in\R$ and $n\in\N.$ To find TWS for it we have to study the
third order ODE
\begin{align*}
-&cU'+a(U^n)'+(U^n)'''+\varepsilon \nabla
g(U,U',-cU',\varepsilon)\cdot (U',U'',-cU'',0)^t\\&=
\big(-cU+aU^n+(U^n)''+\varepsilon g(U,U',-cU',\varepsilon)\big)'\\&=
(-cU+aU^n+n(n-1)U^{n-2}U'+nU^{n-1}U''+\varepsilon
g(U,U',-cU',\varepsilon)\big)'=0.
\end{align*}
Thus, we need to find solutions of the second order ODE
\[
-cU+aU^n+n(n-1)U^{n-2}U'+nU^{n-1}U''+\varepsilon
g(U,U',-cU',\varepsilon)=k
\]
with $k\in\R.$ It writes as the planar system
\[
  \left\{\!
   \begin{array}{l}
     x'=y, \\[10pt] y'=\dfrac{k+cx-ax^n-n(n-1)x^{n-2}y^2}{nx^{n-1}}+\varepsilon
     \dfrac{g_c(x,y,\varepsilon)}{nx^{n-1}},
   \end{array}
  \right.
\]
where $g_c(x,y,\varepsilon)=-g(x,y,-cy,\varepsilon).$  With the new
time $\tau,$  where $d \tau/d s= s_{c,k}(x,y)$ and
$s_{c,k}(x,y)=x^{2(1-n)}/n,$ we get $x=U(\tau)$ satisfies the
equivalent planar ODE
\begin{equation*}
 \begin{cases}
     \dot x=\phantom{-}\dfrac{\partial H_{c,k}(x,y)}{\partial y},
     \\[10pt] \dot y =-\dfrac{\partial H_{c,k}(x,y)}{\partial x}
     +\varepsilon x^{n-1} g_c(x,y,\varepsilon),
 \end{cases}
\end{equation*}
where
\[
H_{c,k}(x,y)= \frac{n}2 x^{2(n-1)}y^2-\frac k n
x^n-\frac{c}{n+1}x^{n+1}+\frac{a}{2n}x^{2n}.
\]
By Lemma \ref{l:lemacentre} (see the comment below its statement), if there exists a singular point $(x_*,0)$ such that
$$
\frac{\partial^2 H_{c,k}}{\partial x^2}(x_*,0)
=x_*^{n-2}\left(a\,(2n-1)\,x_*^n-c\,n\,x_*-k\,(n-1)\right)>0,
$$then it is a center. Furthermore, since the Hypothesis of Proposition \ref{pr:zeroes} are satisfied, we can apply Theorem \ref{th:main2} and the periodic TWS for the
perturbed PDE correspond to simple zeroes of
\[
M_{c,k}(h)=\int_{\gamma_{c,k}(h)} x^{n-1} g_c(x,y,0)\, dx\]
in a suitable interval of the energy. 
To get examples of perturbations
with several simple zeroes we can apply 
Theorem~\ref{th:general}.


\smallskip

\subsubsection{ Camassa-Holm equation and related PDE} The
Camassa-Holm equation is a model for the propagation of shallow
water waves of moderate amplitude. The horizontal component of the
fluid velocity field at a certain depth within the fluid is 
described by the PDE
\[
u_t+(2\kappa+3u)u_x-2u_xu_{xx}+uu_{xxx}-u_{xxt}=0,
\]
and the parameter $\kappa$ is positive. Constantin and Lannes
derived in \cite{ConLan2009}  a similar PDE for surface waves also with 
moderate amplitude in the shallow water regime,
\[u_t+\big(1+6u-6u^2+12u^3)u_x+28u_xu_{xx}+14uu_{xxx}+u_{xxx}-u_{xxt}=0,\]
see also \cite{GasGey2014}. Similarly, the Degasperis-Procesi equation 
\[
u_t+4uu_x-3u_xu_{xx}-uu_{xxx}-u_{xxt}=0,
\]
which was derived initially only for its integrability properties, has a similar role in hydrodynamics.

In fact, perturbations of the above equations can be written under
the common expression
\begin{multline*}
u_t+A'(u)u_x+bu_x u_{xx}+d u
u_{xxx}\\+pu_{xxx}+qu_{xxt}+ru_{xtt}+su_{ttt}+\varepsilon \nabla
g(u,u_x,u_t,\varepsilon)\cdot (u_x,u_{xx},u_{xt},0)^t=0,
\end{multline*}
where $A$ is sufficiently smooth and $b,d,p,q,r$ and $s$ are real
parameters. Its associated  third order ODE is
\begin{align*}
&-cU'+A'(U)U'+bU'U''\\&\hspace{4cm}+ dUU'''+CU'''+\varepsilon \nabla
g(U,U',-cU',\varepsilon)\cdot (U',U'',-cU'',0)^t\\
&=\Big(A_c(U)+b(U')^2/2+ d\big( UU''-(U')^2/2\big) +C
U''+\varepsilon g(U,U',-cU',\varepsilon) \Big)'=0
\end{align*}
where $A_c(U)=A(U)-cU,$ with $A_c(0)=0,$ and  $C=p-qc+rc^2-sc^3.$
Hence, for any function $U$ satisfying the previous equation, there
exists $k\in\R,$ such that
\[
A_c(U)+\beta (U')^2+(C+ d U)U''  +\varepsilon
g(U,U',-cU',\varepsilon)=k,
\]
where $\beta=(b-d)/2.$ The above equation can be written as the planar system

\begin{equation*}
 \begin{cases}
     x'=y,
     \\[10pt] y'=\dfrac{k-A_c(x)-\beta y^2+\varepsilon g_c(x,y,\varepsilon)}{C+dx},
 \end{cases}
\end{equation*}
where $g_c(x,y,\varepsilon)= -g(U,U',-cU',\varepsilon).$ Then,
taking $d\tau/d s= s_{c}(x),$
\[
s_{c}(x)=\begin{cases} (C+dx)^{-2\beta/d} \quad & \mbox{when}\quad
d\ne0,\\
e^{-2\beta x/C} \quad & \mbox{when}\quad d=0,
\end{cases}
\]
 we get
\begin{equation*}
 \begin{cases}
     \dot x=\phantom{-}\dfrac{\partial H_{c,k}(x,y)}{\partial y},
     \\[10pt] \dot y=-\dfrac{\partial H_{c,k}(x,y)}{\partial x}
     +\varepsilon \dfrac{g_c(x,y,\varepsilon)}{(C+dx) s_{c}(x)},
 \end{cases}
\end{equation*}
with
\[
H_{c,k}(x,y)= \frac{y^2}{2s_c(x)}+\int_0^x  \dfrac{A_c(w)-k}{(C+dw) s_{c}(w)} \, dw.
\]
By Lemma \ref{l:lemacentre},  any singular point $(x_*,0)$ such that
$H_{xx}(x_*,0)>0,$ is a center. So by Theorem \ref{th:main2}, the periodic TWS of the perturbed equations correspond with the 
simple zeros of
\[
M_{c.k}(h)=\int_{\gamma_{c,k}(h)}  \dfrac{g_c(x,y,0)}{(C+dx)
s_{c}(x)} \, dx.
\] Again, for some particular examples, the zeroes of the above type of Abelian integrals can be obtained
by using Theorem \ref{th:general}. For instance, we observe that this is trivially the case if $g_c(x,y,0)=(C+dx) s_{c}(x)\left(\sum_{i=0}^\ell d_{2i+1} y^{2i+1}\right)$.

\subsubsection{Boussinesq-type equations} The Boussinesq
equation  describes bi-directional surface water waves and reads 
\[
u_{tt}+uu_{xx}-u_{xx}+(u_x)^2-u_{xxxx}=0.
\]
Similarly, the modified Boussinesq equation is
\[
u_{tt}+uu_{xx}-u_{xx}+(u_x)^2-u_{xxtt}=0,
\]
and appears in the modelling of non-linear waves in a weakly
dispersive medium. We consider the following perturbation of the
family of PDE
\begin{multline}\label{eq:bu}
au_{xx}+bu_{xt}+du_{tt}+2e(uu_{xx}+(u_x)^2)\\+pu_{xxxx}+qu_{xxxt}+ru_{xxtt}+su_{xttt}+fu_{tttt}+\varepsilon
G=0,
\end{multline},
where  $a, b,d,e, p,q,r, s,f$ are suitable real
parameters. We  do not detail here the perturbation $G,$
but it is a function of all the partial derivatives of $u$ up to order four, and such that
after replacing $u$ by $U(x-ct)$ it holds that there exists a
function $g_c$ such that $G=\big(g_c(U,U',\varepsilon)\big)''.$
Hence the ODE associated to \eqref{eq:bu} is
\[
C U''+ e (U^2)''+D U''''+\varepsilon
\big(g_c(U,U',\varepsilon)\big)''= \big(  C U+ e U^2+D
U''+\varepsilon g_c(U,U',\varepsilon)  \big)''=0,
\]
where $C=a-bc+d^2c,$ $D=p-qc+rc^2-sc^3+fc^4,$  and we have used that
$(u^2)_{xx}=2uu_x+2(u_x)^2.$ We are interested in solutions of the
above fourth order ODE
\begin{equation}\label{eq:fin}
 C U+ e U^2+D
U''+\varepsilon g_c(U,U',\varepsilon)=k,
\end{equation}
for some $k\in\R.$ When $D\ne0$ we are again under the situation
covered by Theorem \ref{th:main2}. Notice that other solutions would
satisfy $ C U+ e U^2+D U''+\varepsilon
g_c(U,U',\varepsilon)=k_1s+k_2,$ for some $k_1\ne0,k_2\in\R,$ but we do
not consider them.
 In fact, from \eqref{eq:fin}  we arrive at the same ODE that appears in
the study done in Section \ref{sss:kdv} about the perturbed generalized
Korteweg-de Vries equation, but with a different notation. Indeed,
the above ODE is the same as \eqref{eq:rep} and it can be studied
to get TWS for \eqref{eq:bu} exactly like in that case.

\bigskip

\subsection*{Acknowledgements} The first and third authors are supported by Ministry of Science and Innovation--Research Agency of the Spanish
Government by grants PID2019-104658GB-I00 (first author) and DPI2016-77407-P
 (AEI/FEDER, UE, third author) and
 by the grants AGAUR,  Generalitat de Catalunya
(2017-SGR-1617,  first author) and  (2017-SGR-388, third author).

Part of this work was carried out at the Erwin Schr\"odinger
International Institute for  Mathematics and Physics where  authors
participated in a Research in Teams project in 2018.

\bibliographystyle{acm}

\end{document}